\documentclass{article}
\usepackage{amsmath}
\usepackage{amsthm}
\usepackage{amsfonts}
\usepackage{latexsym}

\newcounter{alphthm}
\newtheorem{thm}{Theorem}[section]
\newtheorem{lem}[thm]{Lemma}
\newtheorem{cor}{Corollary}[section]

\theoremstyle{definition}

\newcommand{\be}{\begin{equation}}
\newcommand{\ee}{\end{equation}}
\newcommand{\ben}{\begin{enumerate}}
\newcommand{\een}{\end{enumerate}}

\title{Two Families  of Finsler Metrics Projectively Related to a Kropina Metric}
\author{A. Tayebi,  H. Sadeghi and  E. Peyghan }
\begin{document}

\maketitle
\begin{abstract}
In this paper, we find the necessary and sufficient conditions under which two classes of $(q, \alpha, \beta)$-metrics   are projectively related  to a Kropina metric.
\end{abstract}

\section{Introduction}
Two regular metrics are called projectively related if there is a diffeomorphism between them such that the pull-back metric is pointwise projective to another one. In Riemannian geometry, two Riemannian metrics $\alpha$ and ${\bar \alpha}$ on a manifold $M$ are projectively related if and only if their spray coefficients have the relation $G^i_\alpha={\bar G}^i_{\bar \alpha}+P_0y^i$, where $P=P(x)$ is a scalar function on $M$ and $P_0:=P_{x^k}y^k$. In Finsler geometry, two Finsler metrics $F$ and  ${\bar F}$ on a  manifold $M$ are called projectively related if $G^i= {\bar G}^i+Py^i$, where   $G^i$ and ${\bar G}^i$ are the geodesic spray coefficients of $F$ and  ${\bar F}$,  respectively and  $P=P(x, y)$ is a scalar function on  the slit tangent bundle $TM_0$. In this case, any geodesic of the first is also geodesic for the second and vice versa.

In order to find explicit examples of projectively related Finsler metrics, we consider $(\alpha, \beta)$-metrics. An $(\alpha,\beta)$-metric is defined by $F:=\alpha\phi(s)$, $s=\beta/\alpha$ where  $\phi=\phi(s)$ is a $C^\infty$ scalar function on $(-b_0, b_0)$ with certain regularity, $\alpha=\sqrt{a_{ij}(x)y^iy^j}$ is a Riemannian metric and  $\beta=b_i(x)y^i$  is a 1-form on a manifold $M$. Thus a natural question arises:
\begin{center}
{\it Under which conditions, two $(\alpha,\beta)$-metrics are projectively related?}
\end{center}
The projective changes between two special $(\alpha,\beta)$-metrics have been studied by many geometers. For example,  Shen has been studied the projectively related Einstein-Finsler metrics \cite{Sh}. A Randers metric $F=\alpha+\beta$  on a manifold $M$ is just a Riemannian metric $\alpha=\sqrt{a_{ij}y^iy^j}$ perturbated by a one form $\beta=b_i(x)y^i$ on $M$ such that  $\|\beta\|_{\alpha}<1$ \cite{TP}.  Then Shen-Yu  studied projectively related Randers metrics \cite{SY}. By the same method,  Cui-Shen find  necessary and sufficient conditions under which the Berwald  metric  $F=\frac{(\alpha+\beta)^2}{\alpha}$ and a Randers metric ${\bar F}=\bar{\alpha}+\bar{\beta}$  are projectively related \cite{CuSh}. Later on,  Zohrevand-Rezaii  do the same for a Matsumoto metric  $F=\frac{\alpha^2}{\alpha-\beta}$ and a Randers metric \cite{ZR}.  Recently,  Chen-Cheng find  necessary and sufficient conditions under which  the metrics in the form $F=\frac{(\alpha+\beta)^p}{\alpha^{p-1}}$ are projectively related to a Randers metric  \cite{CC}. If we substitute $\beta$ with $-\beta$ and take $p = -1$, then we get the Matsumoto metric  which was introduced by Matsumoto as a realization of Finsler's idea  ``a slope measure of a mountain with respect to a time measure" \cite{Matsumoto}\cite{SS}\cite{TPS2}.

There is an important $(\alpha, \beta)$-metric, called Kropina metric $\bar{F}=\frac{\alpha^2}{\beta}$. Kropina metrics were first introduced by L. Berwald in connection with a two-dimensional Finsler space with
rectilinear extremal and were investigated by V. K. Kropina \cite{K}. In \cite{MC}, Mu-Cheng  get the conditions  that a Randers-Kropina metric $F=\alpha+\epsilon\beta+\kappa \alpha^2/\beta$ is projectively equivalent to a Kropina metric $F=\frac{\alpha^2}{\beta}$.  Then, the authors find necessary and sufficient conditions under which a family of Finsler metrics  in the form $F=\frac{\beta^p}{(\beta-\alpha)^{p-1}}$  $(p\neq  1, -1)$ are projectively related to a Randers metric ${\bar F}=\bar{\alpha}+\bar{\beta}$ \cite{TPS1}.

There exists a special subclass of $(\alpha,\beta)$-metrics, namely $(q, \alpha, \beta)$-metrics. Let $\phi:[-1, 1]\rightarrow \mathbb{R}$, $\phi(s)=(1+s)^q$,  $1\leq q\leq 2$ and $||\beta||_{\alpha}<1$. It is easy to see that
\begin{eqnarray*}
\phi'=q(1+s)^{q-1},  \ \ \ \phi''=q(q-1)(1+s)^{q-2}>0.
\end{eqnarray*}
Since $\phi(s)=(1+s)^q>0$, then $\phi-s\phi'=(1+s)^{q-1}[1+s(1-q)]>0, \ \ \  (|s|<1)$. Thus $F:=\alpha\phi(\frac{\beta}{\alpha})=\frac{(\alpha+\beta)^q}{\alpha^{q-1}}$ is a Finsler metric. We call it
$(q, \alpha, \beta)$-metric. When $q = 1$ or  $q = 2$, $F$ becomes Randers metric and Berwald metric, respectively. If we
substitute $\beta$ with $-\beta$ and take $q = -1$, the resulting metric is  Matsumoto
metric.

In this paper,  we are going to find the conditions under which on a manifold $M$ of dimension $n\geq 3$, the $(q, \alpha, \beta)$-metric $F=\frac{(\alpha+\beta)^q}{\alpha^{q-1}}$ and a Kropina  metric $\bar{F}=\frac{{\bar\alpha}^2}{{\bar\beta}}$  being  projectively related. More precisely, we prove the following.
\begin{thm}\label{mainthm1}
Let $F=\frac{(\alpha+\beta)^q}{\alpha^{q-1}}$ $(q\neq 1)$ be a $(q,\alpha,\beta)$-metric and $\bar{F}=\frac{{\bar\alpha}^2}{\bar{\beta}}$ be a Kropina metric on a $n$-dimensional
manifold $M$ $(n \geq3)$ where $\alpha$ and $\bar{\alpha}$ are two Riemannian metrics, $\beta$ and $\bar{\beta}$ are two non-zero collinear 1-forms. Then $F$ is projectively related to $\bar{F}$ if and only if they are Douglas metrics and the geodesic coefficients of $\alpha$ and $\bar{\alpha}$ have the following relation
\begin{eqnarray}
\nonumber G^i_{\alpha}-\bar{G}^i_{\bar{\alpha}}=\!\!\!\!&-&\!\!\!\!\! \frac{1}{2}\frac{q(q-1)\alpha^2 r_{00}}{(1-q^2)\beta^2+(2-q)\alpha\beta +[1+(q^2-q)b^2]\alpha^2} b^i\\\!\!\!\!&+&\!\!\!\!\! \frac{1}{2\bar{b}^2}(\bar{\alpha}^2\bar{s}^i+\bar{r}_{00}\bar{b}^i)+\theta y^i\label{DD11}
\end{eqnarray}
where $b^i:=a^{ij}b_j$, $\bar{b}^i:=\bar{a}^{ij}\bar{b}_j$, $\bar{b}^2:=\|\bar{\beta}\|_{\bar{\alpha}}$, $\tau:=\tau(x)$ is a scalar function and $\theta:=\theta_i y^i$ is a 1-form on $M$.
\end{thm}

\bigskip

Let us define $\phi(s):=s(\frac{s}{s-1})^{q-1}$. By a simple calculation, we get $\phi-s\phi'>0$. Then  $F=\frac{\beta^q}{(\beta-\alpha)^{q-1}}$ is a Finsler metric. This metric is also a special $(q, \alpha, \beta)$-metric. If $q=2$, then $F=\frac{\beta^2}{\beta-\alpha}$  is called \emph{infinite series metric}. Indeed, Let us consider the r-th series $(\alpha,\beta)$-metric $F=\beta\sum_{k=0}^{r}(\frac{\alpha}{\beta})^k$, where we assume  $\alpha<\beta$. If $r=1$, then $F=\alpha+\beta$ is a Randers metric. If we put $r=\infty$, then  we get infinite series metric. We have not at all investigated the geometrical meaning about the infinite series metric by this time. But this metric is remarkable as the difference between a Randers metric  and a Matsumoto metric.

\begin{thm}\label{mainthm3}
Let $F=\frac{\beta^q}{(\beta-\alpha)^{q-1}}$ $(q\neq 1,-1)$   be a $(q,\alpha,\beta)$-metric and $\bar{F}=\frac{\bar{\alpha^2}}{\bar{\beta}}$ be a Kropina metric on a $n$-dimensional
manifold $M$ $(n \geq3)$ where $\alpha$ and $\bar{\alpha}$ are two Riemannian metrics, $\beta$ and $\bar{\beta}$ are two non-zero collinear 1-forms. Then $F$ is projectively related to $\bar{F}$ if and only if they are Douglas metrics and the geodesic coefficients of $\alpha$ and $\bar{\alpha}$ have the following relation
\begin{eqnarray}
G^i_{\alpha}-\bar{G}^i_{\bar{\alpha}}=\frac{1}{2\bar{b}^2}(\bar{\alpha}^2\bar{s}^i+\bar{r}_{00}\bar{b}^i)+\theta y^i-\frac{ q\alpha^3 r_{00}}{2\big[\beta^2(\beta-\alpha)+q(b^2\alpha^2-\beta^2)\alpha\big]} b^i,\label{p7}
\end{eqnarray}
where $b^i:=a^{ij}b_j$, $\bar{b}^i:=\bar{a}^{ij}\bar{b}_j$, $\bar{b}^2:=\|\bar{\beta}\|_{\bar{\alpha}}$,  $\tau:=\tau(x)$ is a scalar function and $\theta:=\theta_i y^i$ is a 1-form on $M$.
\end{thm}

%--------------------------------------------------------------------------------------------------------------
\section{Preliminary}
%--------------------------------------------------------------------------------------------------------------
  An $(\alpha, \beta)$-metric is a Finsler metric on a manifold $M$ defined by $F:=\alpha\phi(s)$, where $s=\beta/\alpha$,  $\phi=\phi(s)$ is a $C^\infty$ function on the $(-b_0, b_0)$ with certain regularity, $\alpha=\sqrt{a_{ij}y^iy^j}$ is a Riemannian metric and $\beta=b_i(x)y^i$ is a 1-form on $M$.  For an $(\alpha, \beta)$-metric, let us define $b_{i|j}$ by $b_{i|j}\theta^j:=db_i-b_j\theta^j_i$, where $\theta^i:=dx^i$ and $\theta^j_i:=\Gamma^j_{ik}dx^k$ denote
the Levi-Civita connection form of $\alpha$. Let
\begin{eqnarray*}
r_{ij}:=\frac{1}{2}(b_{i|j}+b_{j|i}), \ \ \ s_{ij}:=\frac{1}{2}(b_{i|j}-b_{j|i}).
\end{eqnarray*}
Clearly, $\beta$ is closed if and only if $s_{ij}=0$. An $(\alpha,\beta)$-metric is said to be trivial if $r_{ij}=s_{ij}=0$. Put
\begin{eqnarray*}
&&r_{i0}: = r_{ij}y^j, \  \ r_{00}:=r_{ij}y^iy^j, \ \ r_j := b^i r_{ij},\\
&&s_{i0}:= s_{ij}y^j, \  \ \ s_j:=b^i s_{ij},\\
&& r_0:= r_j y^j,\ \  \ \  s_0 := s_j y^j.
\end{eqnarray*}
For  an $(\alpha,\beta)$-metric $F=\alpha\phi(s)$, $s=\frac{\beta}{\alpha}$, if we put
\[
Q:=\frac{\phi'}{\phi-s\phi'}
\]
then
\begin{eqnarray*}
&&Q'=\frac{\phi\phi''}{(\phi-s\phi')^2}, \ \ \ Q''=\frac{\phi'\phi''+\phi\phi'''}{(\phi-s\phi')^2}+\frac{2s\phi{\phi''}^2}{(\phi-s\phi')^3}.
\end{eqnarray*}
Now, let $\phi=\phi(s)$ be a positive $C^{\infty}$ function on
$(-b_0,b_0)$. For a number $b\in[0,b_0)$, let
\begin{eqnarray}
\Delta:=1+sQ+(b^2-s^2)Q'.\label{del1}
\end{eqnarray}
Let $G^i=G^i(x,y)$ and $\bar{G}^i_{\alpha}=\bar{G}^i_{\alpha}(x,y)$ denote the
coefficients  of $F$ and  $\alpha$ respectively in the same coordinate system. By definition, we have
\begin{eqnarray}
G^i=G^i_{\alpha}+\alpha Q s^i_0+(-2Q\alpha s_0+r_{00})(\Theta\frac{y^i}{\alpha}+\Psi b^i),\label{GG01}
\end{eqnarray}
where
\begin{eqnarray*}
&&\Theta:=\frac{Q-sQ'}{2\Delta}={\phi\phi' -s (\phi\phi'' +\phi'\phi')\over 2 \phi \Big[  ( \phi -s\phi')+( b^2 -s^2)\phi''  \Big ]}\\
&&\Psi:=\frac{Q'}{2\Delta}={1\over 2} { \phi'' \over  ( \phi -s \phi')+( b^2 -s^2)\phi''}.
\end{eqnarray*}
By (\ref{G1}), it follows that every trivial $(\alpha,\beta)$-metric satisfies $G^i=G^i_{\alpha}$ and then it reduces to a Berwald metric.

%--------------------------------------------------------------------------------------------------------------------
\section{Proof of Theorem \ref{mainthm1}}
%--------------------------------------------------------------------------------------------------------------------
For an $(q,\alpha,\beta)$-metric $F=\frac{(\alpha+\beta)^q}{\alpha^{q-1}}$,
the following are hold
\begin{eqnarray}
&&\nonumber Q=\frac{q}{s(1-q)+1},\\
&&\nonumber\Theta=\frac{1}{2}\frac{q(1-2(q-1)s)}{s^2(1-q^2)+s(2-q)+1+b^2q(q-1)} ,\\
&&\Psi:=\frac{1}{2}\frac{q(q-1)}{s^2(1-q^2)+s(2-q)+1+b^2q(q-1)}.\label{Q02}
\end{eqnarray}
For a Kropina metric $\bar{F}=\bar{\alpha}+\bar{\beta}$, we have
\begin{eqnarray}
&&\nonumber\bar{Q}:=-\frac{1}{2s},\\
&&\nonumber\bar{\Theta}:=-\frac{s}{2\bar{b}^2},\\
&&\bar{\Psi}:=\frac{1}{2\bar{b}^2}\label{Q03}.
\end{eqnarray}
The geodesic curves of a Finsler metric $F=F(x,y)$ on a smooth manifold $M$, are determined  by the system of second order differential equations
\[
\frac{d^2 x^i}{dt^2}+2G^i\big(x,\frac{dx}{dt}\big)=0,
\]
where the local functions $G^i=G^i(x, y)$ are called the  spray coefficients, and given by
\[
G^i=\frac{1}{4}g^{il}\Big\{\frac{\partial^2 F^2}{\partial x^k\partial y^l}y^k-\frac{\partial F^2}{\partial x^l}\Big\}.
\]
A Finsler metric $F$ is  called a Berwald metric, if  $G^i$  are quadratic in $y\in T_xM$  for any $x\in M$.

Let
\begin{eqnarray}
D^i_{j\ kl}:=\frac{\partial^3}{\partial y^j\partial y^k\partial
y^l}\big(G^i-\frac{1}{n+1}\frac{\partial G^m}{\partial
y^m}y^i\big).\label{Q1}
\end{eqnarray}
 It is easy to verify that
$\mathcal{D}:=D^i_{j\ kl}dx^j\otimes\partial_i\otimes dx^k\otimes
dx^l$ is a well-defined tensor on slit tangent bundle $TM_0$. We
call $\mathcal{D}$ the Douglas tensor. The Douglas tensor
$\mathcal{D}$ is a non-Riemannian projective invariant, namely, if
two Finsler metrics $F$ and $\bar{F}$ are projectively equivalent,
$G^i=\bar{G}^i+Py^i$, where $P=P(x, y)$ is positively
$y$-homogeneous of degree one, then the Douglas tensor of $F$ is
same as that of  $\bar{F}$.  Finsler metrics with vanishing
Douglas tensor are called Douglas metrics \cite{NST}\cite{NT}\cite{TP2}\cite{TP3}. The notion of Douglas metrics was first proposed by B$\acute{a}$cs$\acute{o}$-Matsumoto as a generalization  of Berwald metrics \cite{BM3}.

To prove Theorem \ref{mainthm1}, we remark the following.

\begin{lem}\label{lem1}{\rm \cite{MD}}
\emph{Let $F=\frac{\alpha^2}{\beta}$ is a Kropina metric on a n-dimensional manifold $M$. Then\\
(1) $(n\geq 3)$ Kropina metric $F$ with $(b^2\neq 0)$ is a Douglas metric if and only if
\begin{eqnarray}
\bar{s}_{ij}=\frac{1}{\bar{b}^2}(\bar{b}_i\bar{s}_j-\bar{b}_j\bar{s}_i);
\end{eqnarray}
 (2) $(n= 2)$ Kropina metric $F$ is a Douglas metric.}
\end{lem}

\bigskip

For an $(\alpha,\beta)$-metric , the Douglas tensor is determined
by
\be
D^i_{j\ kl}:=\frac{\partial^3}{\partial y^j\partial y^k\partial y^l}\big(T^i-\frac{1}{n+1}\frac{\partial T^m}{\partial
y^m}y^i\big),\label{D1}
\ee
where
\be
T^i:= \alpha Q s^i_0+\Psi(r_{00}-2\alpha Q s_0)b^i\label{QA},
\ee
and
\be
T^m_{y^m}= Q's_0+\Psi'\alpha^{-1}(b^2-s^2)(r_{00}-2\alpha Q
s_0)+2\Psi\big[r_0-Q'(b^2-s^2)s_0-Qss_0\big]\label{Q3}.
\ee
Now, let $F$ and $\bar F$ be two $(\alpha,\beta)$-metrics which  have the same Douglas tensor, i.e., $D^i_{jkl}=\bar D^i_{jkl}$. From (\ref{Q1}) and (\ref{D1}), we have
\begin{eqnarray}
\frac{\partial^3}{\partial y^i\partial y^j\partial
y^k}\Big[T^i-{\bar T^i}-\frac{1}{n+1}(T^m_{y^m}-{\bar
T^m_{y^m}})y^i\Big]=0.\label{Q2}
\end{eqnarray}
Then there exists a class of scalar function
$H^i_{jk}:=H^i_{jk}(x)$ such that
\begin{eqnarray}
T^i-{\bar T^i}-\frac{1}{n+1}(T^m_{y^m}-{\bar
T^m_{y^m}})y^i=H^i_{00},\label{Q01}
\end{eqnarray}
where $H^i_{00}=H^i_{jk}(x)y^iy^j$, $T^i$ and $T^m_{y^m}$ are given by (\ref{QA}) and (\ref{Q3}) respectively. In this paper,  we assume that $\lambda:=\frac{1}{n+1}$.

\bigskip

\begin{lem}\label{lemma1}
Let $F=\frac{(\alpha+\beta)^q}{\alpha^{q-1}}$ $(q\neq 1)$  be a $(q,\alpha,\beta)$-metric and $\bar{F}=\frac{\bar{\alpha^2}}{\bar{\beta}}$ be a Kropina metric on a $n$-dimensional
manifold $M$ $(n \geq3)$,  where $\alpha$ and $\bar{\alpha}$ are two Riemannian metrics and $\beta$ and $\bar{\beta}$ are two non-zero collinear 1-forms. Then $F$ and $\bar{F}$ have the same Douglas tensor if and only if they are all Douglas metrics.
\end{lem}
\begin{proof}
The sufficiency is obvious. Suppose that $F$ and $\bar{F}$ have the same Douglas
tensor on an n-dimensional manifold $M$ when $n\geq 3$. Then ($\ref{Q01}$) holds. Plugging
(\ref{Q02}) and (\ref{Q03}) into (\ref{Q01}), we obtain
\begin{eqnarray}
\frac{A^i\alpha^6+B^i\alpha^5+C^i\alpha^4+D^i\alpha^3+E^i\alpha^2+F^i\alpha+H^i}{I\alpha^5+J\alpha^4+K\alpha^3+L\alpha^2+M\alpha+N}+\frac{\bar{A}^i\bar{\alpha}^2+\bar{B}^i}{2\bar{b}^2\bar{\beta}}=H^i_{00},\label{Q04}
\end{eqnarray}
where
\begin{eqnarray*}
A^i\!\!\!\!&:=&\!\!\!\!\! -2q(1-qb^2+q^2b^2)\Big[(1-qb^2+q^2b^2)s^i_0-q(q-1)s_0b^i\Big],\\
B^i\!\!\!\!&:=&\!\!\!\!\! q\Big[2(p-1)\lambda (1+qb^2)s_0y^i-2q(q-1)(q-2)\beta s_0b^i\\
&&+4\beta(q-2)(q^2b^2-qb^2+1)s^i_0+2\lambda(q-1)(q^2b^2-qb^2+1)r_0y^i\\
&&-(q-1)(q^2b^2-qb^2+1)r_{00}b^i\Big], \\
C^i\!\!\!\!&:=&\!\!\!\!\! q\Big[(q-1)(q^3 b^2-2q^2b^2+qb^2+2q-3)\beta r_{00}b^i+\lambda (q-1)(q-2)b^2r_{00}y^i\\
&&\ -2(q-1)\lambda(q^3 b^2-2q^2b^2+qb^2+2q-3)\beta r_{0}y^i\\
&&+\big[4qb^2(q+1)(q-1)^2+(2q^2+8q-12)\big]\beta^2s^i_0\\
&&-2\lambda (q-1)(3q^3b^2-2b^2q^2+2q-b^2q-3)\beta s_0y^i\\
&&-2q(q+1)(q-1)^2\beta^2s_0b^i\Big],\\
D^i\!\!\!\!&:=&\!\!\!\!\! q(q-1)\beta\Big[(2q^2-10q+6)\lambda\beta s_0y^i-4(q+1)(q-2)\beta^2s^i_0\\
&&-(q+4)(q-1)b^2r_{00}y^i+3(q-1)\beta r_{00}b^i\Big],\\
E^i\!\!\!\!&:=&\!\!\!\!\! -q(q-1)\beta^2\Big[2(q-1)(q+1)^2\beta^2s^i_0+\lambda \big[2b^2(q+1)(q-1)^2+(q-2)\big]r_{00}y^i\\
&&-2\lambda(q-1)(3q-1)(q+1)\beta s_0y^i-2\lambda (q+1)(q-1)^2r_0y^i\\
&&+(q+1)(q-1)^2\beta r_{00}b^i\Big],\\
F^i\!\!\!\!&:=&\!\!\!\!\! -\lambda q(q+4)(q-1)^2 \beta^3 r_{00}y^i,\\
H^i\!\!\!\!&:=&\!\!\!\!\! 2\lambda q(q+1)(p-1)^3 \beta^4r_{00}y^i.
\end{eqnarray*}
and
\begin{eqnarray*}
I\!\!\!\!&:=&\!\!\!\!\! -2(-qb^2+q^2b^2+1)^2,\\
J\!\!\!\!&:=&\!\!\!\!\! 2\beta(-qb^2+q^2b^2+1)(qb^2-2q^2b^2+q^3b^2+3q-5),\\
K\!\!\!\!&:=&\!\!\!\!\! 2\beta^2(-10-q^2+10q+6q^3b^2-12q^2b^2+6qb^2),\\
L\!\!\!\!&:=&\!\!\!\!\! -2(q-1)\beta^3(2q^4 b^2-2q^3b^2+3q^2-2q^2b^2+2q+2qb^2-10),\\
M\!\!\!\!&:=&\!\!\!\!\! 2\beta^4(q+1)(q-5)(q-1)^2,\\
N\!\!\!\!&:=&\!\!\!\!\! 2\beta^5(q+1)^2(q-1)^3.
\end{eqnarray*}
and
\begin{eqnarray*}
&&\bar{A}^i:=\bar{b}^2\bar{s}^i_0-\bar{b}^i\bar{s}_0,\\
&&\bar{B}^i:=\bar{\beta}[2\lambda y^i(\bar{r}_0+\bar{s}_0)-\bar{b}^i\bar{r}_{00}]
\end{eqnarray*}
(\ref{Q04}) is equivalent to following
\begin{eqnarray}
\nonumber 2\bar{b}^2\bar{\beta}(A^i\alpha^6\!\!\!\!&+&\!\!\!\!\! B^i\alpha^5+C^i\alpha^4+D^i\alpha^3+E^i\alpha^2+F^i\alpha+H^i)
\\ \nonumber
\!\!\!\!&+&\!\!\!\!\!(\bar{A}^i\bar{\alpha}^2+\bar{B}^i)(I\alpha^5+J\alpha^4+K\alpha^3+L\alpha^2+M\alpha+N)
\\
\!\!\!\!&=&\!\!\!\!\!\ \ 2\bar{b}^2\bar{\beta}(I\alpha^5+J\alpha^4+K\alpha^3+L\alpha^2+M\alpha+N)H^i_{00}.\label{Q06}
\end{eqnarray}
First we show that $\bar{A}^i$ can be divide by $\bar{\beta}$.

By replacing $y^i$ with $-y^i$ in (\ref{Q06}), we get the following
\begin{eqnarray}
\nonumber-2\bar{b}^2\bar{\beta}(-A^i\alpha^6\!\!\!\!&+&\!\!\!\!\! B^i\alpha^5-C^i\alpha^4+D^i\alpha^3-E^i\alpha^2+F^i\alpha-H^i)\\
\nonumber\!\!\!\!&-&\!\!\!\!\! (\bar{A}^i\bar{\alpha}^2+\bar{B}^i)(I\alpha^5-J\alpha^4+K\alpha^3-L\alpha^2+M\alpha-N)\\
\!\!\!\!&=&\!\!\!\!\!\ \ -2\bar{b}^2\bar{\beta}(I\alpha^5-J\alpha^4+K\alpha^3-L\alpha^2+M\alpha-N)H^i_{00}.\label{Q07}
\end{eqnarray}
 $(\ref{Q06})+(\ref{Q07})$ yields
\begin{eqnarray}
\nonumber 2\bar{b}^2\bar{\beta}(A^i\alpha^6+C^i\alpha^4+ E^i\alpha^2+H^i)\!\!\!\!&+&\!\!\!\!\!(\bar{A}^i\bar{\alpha}^2+\bar{B}^i)
(J\alpha^4+L\alpha^2+N)\\\!\!\!\!&=&\!\!\!\!\!  \ 2\bar{b}^2\bar{\beta}(J\alpha^4+L\alpha^2+N)H^i_{00}.\label{Q08}
\end{eqnarray}
$(\ref{Q06})-(\ref{Q07})$ implies that
\begin{eqnarray}
\nonumber(B^i\alpha^4+D^i\alpha^2+F^i)(2\bar{b}^2\bar{\beta})\!\!\!\!&+&\!\!\!\!\! (\bar{A}^i\bar{\alpha}^2+\bar{B}^i)(I\alpha^4+K\alpha^2+M)\\
\!\!\!\!&=&\!\!\!\!\! \ 2\bar{b}^2\bar{\beta}(I\alpha^4+K\alpha^2+M)H^i_{00}.\label{Q09}
\end{eqnarray}
If $q=-1$ then $H^i=N=M=0$. Thus (\ref{Q08}) and (\ref{Q09}) are equivalent to
\begin{eqnarray}
2\bar{b}^2\bar{\beta}(A^i\alpha^4+C^i\alpha^2+E^i)+(\bar{A}^i\bar{\alpha}^2+\bar{B}^i)(J\alpha^2+L)=2\bar{b}^2\bar{\beta}(J\alpha^2+L)H^i_{00}\label{QQ}
\end{eqnarray}
and
\begin{eqnarray}
2\bar{b}^2\bar{\beta}(B^i\alpha^4+D^i\alpha^2+F^i)+(\bar{A}^i\bar{\alpha}^2+\bar{B}^i)(I\alpha^4+K\alpha^2)
=2\bar{b}^2\bar{\beta}(I\alpha^4+K\alpha^2)H^i_{00}.\label{QQ0}
\end{eqnarray}
By (\ref{QQ}) and (\ref{QQ0}), it results that $(\bar{A}^i\bar{\alpha}^2+\bar{B}^i)(J\alpha^2+L)$ and $(\bar{A}^i\bar{\alpha}^2+\bar{B}^i)(I\alpha^4+K\alpha^2)$ can be divided by $\bar{\beta}$. Thus $\beta=\mu\bar{\beta}$ and $\bar{A}^i\bar{\alpha}^2 I \alpha^4$ can be divided by $\bar{\beta}$. Since $\bar{\beta}$ is prime with respect to $\alpha$ and $\bar{\alpha}$, therefore $\bar{A}^i:=\bar{b}^2\bar{s}^i_0-\bar{b}^i\bar{s}_0$ can be divided by $\bar{\beta}$. If $q\neq 1,-1$,  then (\ref{Q08}) and (\ref{Q09}) implies that $(\bar{A}^i\bar{\alpha}^2+\bar{B}^i)(J\alpha^4+L\alpha^2+N)$ and $(\bar{A}^i\bar{\alpha}^2+\bar{B}^i)(I\alpha^4+K\alpha^2+M)$ can be divided by $\bar{\beta}$. Since $\bar{\beta}$ is prime with respect to $\alpha$ and $\bar{\alpha}$, then $\bar{A}^i:=\bar{b}^2\bar{s}^i_0-\bar{b}^i\bar{s}_0$ can be divided by $\bar{\beta}$. Hence, there is a scaler function $\psi^i(x)$ such that
\begin{eqnarray}
\bar{b}^2\bar{s}^i_0-\bar{b}^i\bar{s}_0=\psi^i\bar{\beta}.\label{QQ2}
\end{eqnarray}
Contracting (\ref{QQ2}) with $\bar{y}_i:=\bar{a}_{ij}y^j$ yields
\[
\psi^i(x)=-\bar{s}^i.
\]
Then we have
\begin{eqnarray}
\bar{s}_{ij}=\frac{1}{\bar{b}^2}(\bar{b}_i\bar{s}_j-\bar{b}_j\bar{s}_i).\label{O1}
\end{eqnarray}
Now, suppose that $(n\geq 3)$. Then by Lemma \ref{lem1}, $\bar{F} =\frac{\bar{\alpha}^2}{\bar{\beta}}$ is a Douglas metric. Since $F$ and $\bar{F}$ have the same Douglas tensor, then both of them are Douglas metrics.

If $(n=2)$, then $\bar{F}=\frac{\bar{\alpha}^2}{\bar{\beta}}$ is a Douglas metric by Lemma \ref{lem1}. Thus $F$ and $\bar{F}$ having  the same Douglas tensors. This means that they are all Douglas metrics. This completes the
proof of lemma $\ref{lemma1}$.
\end{proof}

\bigskip

On the other hand, the following holds.
\begin{lem}\label{lemma3}{\rm \cite{LSS}}
\emph{Suppose that $\frac{Q}{s}\neq constant$ for an $(\alpha,\beta)-$ metric $F=\phi(\frac{\beta}{\alpha})$ on a manifold $M$ of dimension $n$ $(n\geq 3)$. If $F$ is a Douglas metric and $b:=\|\beta_x\|_{\alpha}\neq 0$, then $\beta$ is closed.}
\end{lem}

\bigskip

Now, we are in the position to prove Theorem $\ref{mainthm1}$.

\bigskip

\noindent
{\bf Proof of Theorem \ref{mainthm1}:} We prove the theorem in two cases, as follows.\\\\
\textbf{Case (1):} When $q=-1$.\\
First we proof the necessity. If $F$ is projectively equivalent
to $\bar{F}$, then  they have the same Douglas tensor. By Lemma $\ref{lemma1}$, $F$ and $\bar{F}$
are both Douglas metrics.  $F=\frac{\alpha^2}{\alpha+\beta}$
is a Douglas metric if and only if $b_{i|j}=0$. Thus by ($\ref{G1}$), we have
\begin{eqnarray}
G^i=G^i_{\alpha}.\label{G0}
\end{eqnarray}
On the other hand, plugging (\ref{O1}) and (\ref{Q03}) into ($\ref{GG01}$) yields
\begin{eqnarray}
\bar{G}^i=\bar{G}^i_{\bar{\alpha}}-\frac{1}{2\bar{b}^2}\big[-\bar{\alpha}^2\bar{s}^i
+(2\bar{s}_0y^i-\bar{r}_{00}\bar{b}^i)+2\frac{\bar{r}_{00}\bar{\beta}y^i}{\bar{\alpha}^2}\big].\label{G11}
\end{eqnarray}
By the projective equivalence of $F$ and $\bar{F}$ again, there is a scalar function $P = P(x,y)$
on $TM_0$ such that $G^i=\bar{G}^i+Py^i$. From (\ref{G0}) and  (\ref{G11}) we have
\begin{eqnarray}
\big[P-\frac{1}{\bar{b}^2}(\bar{s}_0+\frac{\bar{r}_{00}\bar{\beta}}{\bar{\alpha}^2})\big]y^i=G^i_{\alpha}-\bar{G}^i_{\bar{\alpha}}
-\frac{1}{2\bar{b}^2}(\bar{\alpha}^2\bar{s}^i+\bar{r}_{00}\bar{b}^i).\label{DD12}
\end{eqnarray}
Note that the right side of (\ref{DD12}) is a quadratic in $y$. Then there exists a 1-form
$\theta=\theta_i(x)y^i$ on $M$ such that
\begin{eqnarray}
P-\frac{1}{\bar{b}^2}(\bar{s}_0+\frac{\bar{r}_{00}\bar{\beta}}{\bar{\alpha}^2})=\theta.
\end{eqnarray}
Thus we have
\begin{eqnarray}
&&G^i_{\alpha}=\bar{G}^i_{\bar{\alpha}}+\frac{1}{2\bar{b}^2}(\bar{\alpha}^2\bar{s}^i+\bar{r}_{00}\bar{b}^i)+\theta y^i.\label{D11}
\end{eqnarray}
This completes the proof of the necessity.

Conversely, because of $r_{00}=0$ and  from (\ref{G0}), (\ref{G11}) and (\ref{DD11}) we have
\begin{eqnarray}
G^i=\bar{G}^i+\big[\theta+\frac{1}{2\bar{b}^2}(\bar{s}_0+\frac{\bar{r}_{00}\bar{\beta}}{\bar{\alpha}^2})\big]y^i.
\end{eqnarray}
Thus $F$ is projectively equivalent to $\bar{F}$.\\\\
\textbf{Case (2):} When $q\neq 1,-1$.\\
First we proof the necessity. If $F$ is projectively equivalent
to $\bar{F}$ they have the same Douglas tensor. By Lemma $\ref{lemma1}$, we know that $F$ and $\bar{F}$
are both Douglas metrics. If $q=1,-1$, then it is easy to prove that $\phi(s)=(1+s)^q$ satisfies $\frac{Q}{S}\neq constant$. By lemma $\ref{lemma3}$, we have $s_{ij}=0$. By ($\ref{GG01}$), it follows that
\begin{eqnarray}
&&\nonumber G^i=G^i_{\alpha}+\frac{1}{2}\frac{q(\alpha-2(q-1)\beta)r_{00}}{(1-q^2)\beta^2+(2-q)\alpha\beta +(1+q(q-1)b^2)\alpha^2} y^i\\
&&+\frac{1}{2}\frac{q(q-1)\alpha^2 r_{00}}{(1-q^2)\beta^2+(2-q)\alpha\beta +(1+q(q-1)b^2)\alpha^2} b^i.\label{GGG1}
\end{eqnarray}
On the other hand, plugging (\ref{O1}) and (\ref{Q03}) into ($\ref{GG01}$) yields
\begin{eqnarray}
\bar{G}^i=\bar{G}^i_{\bar{\alpha}}-\frac{1}{2\bar{b}^2}\big[-\bar{\alpha}^2\bar{s}^i+(2\bar{s}_0y^i
-\bar{r}_{00}\bar{b}^i)+2\frac{\bar{r}_{00}\bar{\beta}y^i}{\bar{\alpha}^2}\big].\label{GG11}
\end{eqnarray}
By the projective equivalence of $F$ and $\bar{F}$ again, there is a scalar function $P = P(x,y)$
on $TM_0$ such that $G^i=\bar{G}^i+Py^i$. By (\ref{GGG1}) and (\ref{GG11}), we have
\begin{eqnarray}
\nonumber\bar{G}^i_{\bar{\alpha}}\!\!\!\!&-&\!\!\!\!\!G^i_{\alpha}+\big[P- \frac{1}{2\bar{b}^2}(\bar{s}_0+\bar{r}_{00}\bar{b}^i)-\frac{1}{2}\frac{q[\alpha-2(q-1)\beta]r_{00}}
{(1-q^2)\beta^2+(2-q)\alpha\beta +[1+(q^2-q)b^2]\alpha^2}\big]y^i\\
\!\!\!\!&=&\!\!\!\!\!\frac{1}{2}\frac{(q^2-q)\alpha^2 r_{00}}{(1-q^2)\beta^2+(2-q)\alpha\beta +[1+q(q-1)b^2]\alpha^2} b^i-\frac{1}{2\bar{b}^2}(\bar{\alpha}^2\bar{s}^i+\bar{r}_{00}\bar{b}^i).\label{D12}
\end{eqnarray}
Note that the right side of (\ref{D12}) is a quadratic in $y$. Then there exists a 1-form
$\theta=\theta_i(x)y^i$ on $M$ such that
\begin{eqnarray}
P-\frac{1}{2\bar{b}^2}(\bar{s}_0+\bar{r}_{00}\bar{b}^i)-\frac{1}{2}\frac{q(\alpha-2(q-1)\beta)r_{00}}{(1-q^2)\beta^2+
(2-q)\alpha\beta +(1+q(q-1)b^2)\alpha^2}=\theta.
\end{eqnarray}
Thus we get
\begin{eqnarray}
\nonumber G^i_{\alpha}\!\!\!\!&+&\!\!\!\!\!\ \frac{1}{2}\frac{q(q-1)\alpha^2 r_{00}}{(1-q^2)\beta^2+(2-q)\alpha\beta +(1+q(q-1)b^2)\alpha^2} b^i=\bar{G}^i_{\bar{\alpha}}\\
\!\!\!\!&+&\!\!\!\!\!\ \frac{1}{2\bar{b}^2}(\bar{\alpha}^2\bar{s}^i+\bar{r}_{00}\bar{b}^i)+\theta y^i.\label{D11}
\end{eqnarray}
This completes the proof of the necessity.

Conversely, by  (\ref{DD11}), (\ref{G0}) and (\ref{G11})  we have
\begin{eqnarray*}
G^i-\bar{G}^i=\big[\theta+\frac{1}{2\bar{b}^2}(\bar{s}_0+\bar{r}_{00}\bar{b}^i)
+\frac{1}{2}\frac{q[\alpha-2(q-1)\beta]r_{00}}{(1-q^2)\beta^2+(2-q)\alpha\beta +[1+(q^2-q)b^2]\alpha^2}\big]y^i.
\end{eqnarray*}
Thus $F$ is projectively equivalent to $\bar{F}$. This completes the proof.
\qed

\bigskip

By Lemma \ref{lem1}, Lemma \ref{lemma3} and Theorem \ref{mainthm1}, we have the following.
\begin{cor}\label{cor2}
Let $F=\frac{(\alpha+\beta)^q}{\alpha^{q-1}}$ $(q\neq 1)$ be a $(q,\alpha,\beta)$ metric and $\bar{F}=\frac{\bar{\alpha^2}}{\bar{\beta}}$  be a Kropina metric on a $n-$ dimensional
manifold $M$ $(n \geq3)$ where $\alpha$ and $\bar{\alpha}$ are two Riemannian metrics, $\beta$ and $\bar{\beta}$ are two nonzero collinear 1-forms. Then $F$ is projectively equivalent to $\bar{F}$ if and only if .\\
\begin{eqnarray*}
&& G^i_{\alpha}+\frac{1}{2}\frac{q(q-1)\alpha^2 r_{00}}{(1-q^2)\beta^2+(2-q)\alpha\beta +[1+q(q-1)b^2]\alpha^2} b^i=\bar{G}^i_{\bar{\alpha}}+\frac{1}{2\bar{b}^2}(\bar{\alpha}^2\bar{s}^i+\bar{r}_{00}\bar{b}^i)+\theta y^i,\\
&&s_{ij}=0,\\
&&\bar{s}_{ij}:=\frac{1}{\bar{b}^2}\{\bar{b}_i\bar{s}_j-\bar{b}_j\bar{s}_i\}.
\end{eqnarray*}
where $b_{i|j}$ denote the coefficients of the covariant derivatives of $\beta$ with respect to $\alpha$.
\end{cor}

\bigskip

It is well known that the Berwald metric $F=\frac{(\alpha+\beta)^2}{\alpha}$ on a manifold $M$ is a
Douglas metric if and only if
\begin{eqnarray}
b_{i|j}=2\tau [(1+2b^2)a_{ij}-3b_ib_j],\label{010}
\end{eqnarray}
where $\tau=\tau(x)$ is a scalar function on $M$. Thus by (\ref{010})  and  Theorem \ref{mainthm1}, we have the following.
\begin{cor}
Let $F=\frac{(\alpha+\beta)^2}{\alpha}$  be a  Berwald  metric and $\bar{F}=\frac{\bar{\alpha^2}}{\bar{\beta}}$ be a Kropina metric on a $n$-dimensional
manifold $M$ $(n \geq3)$ where $\alpha$ and $\bar{\alpha}$ are two Riemannian metrics, $\beta$ and $\bar{\beta}$ are two non-zero collinear 1-forms. Then $F$ is projectively related to $\bar{F}$ if and only if they are Douglas metrics and the following holds
\begin{eqnarray*}
G^i_{\alpha}-\bar{G}^i_{\bar{\alpha}}=\frac{1}{2\bar{b}^2}(\bar{\alpha}^2\bar{s}^i+\bar{r}_{00}\bar{b}^i)+\theta y^i-2\tau\alpha^2 b^i.
\end{eqnarray*}
\end{cor}

%--------------------------------------------------------------------------------------------------------------------
\section{Proof of Theorem \ref{mainthm3}}
%--------------------------------------------------------------------------------------------------------------------
In this section, we are going to prove the Theorem \ref{mainthm3}. More precisely, we find the conditions that an   $(q,\alpha,\beta)$-metric $F=\frac{\beta^q}{(\beta-\alpha)^{q-1}}$ being projectively equivalent to a Kropina metric. For the $(q,\alpha,\beta)$-metric $F=\frac{\beta^q}{(\beta-\alpha)^{q-1}}$, the following are hold
\begin{eqnarray}
&&\nonumber \phi=\frac{s^p}{(s-1)^{q-1}},\\
&&\nonumber Q=\frac{s-q}{(q-1)s},\\
&&\nonumber \Psi=\frac{q}{2[s^2(s-1)+q(b^2-s^2)]} ,\\
&&\Theta=\frac{s(s-2q)}{2[s^2(s-1)+q(b^2-s^2)]}\label{Q001}.
\end{eqnarray}

\bigskip

First we prove the following.

\begin{lem}\label{lemma2}
Let $F=\frac{\beta^q}{(\beta-\alpha)^{q-1}}$  be an $(q,\alpha,\beta)$-metric  and $\bar{F}=\frac{\bar{\alpha^2}}{\bar{\beta}}$ be a Kropina metric on a $n$-dimensional
manifold $M$ $(n \geq3)$ where $\alpha$ and $\bar{\alpha}$ are two Riemannian metrics and $\beta$ and $\bar{\beta}$ are two non-zero collinear 1-forms. Then  $F$ and $\bar{F}$ have the same Douglas tensor if and only if they are all Douglas metrics.
\end{lem}
\begin{proof}
The sufficiency is obvious. Suppose that $F$ and $\bar{F}$ have the same Douglas
tensor on an n-dimensional manifold $M$ when $n\geq 3$. Then ($\ref{Q01}$) holds. By plugging
(\ref{Q03}) and (\ref{Q001}) into (\ref{Q01}), we obtain
\begin{eqnarray}
\frac{\sum^8_{j=1}A^i_j\alpha^j}{\sum^6_{j=0}B_j\alpha^j}+\frac{\bar{A}^i\bar{\alpha}^2+\bar{B}^i}{2\bar{b}^2\bar{\beta}}=H^i_{00},\label{Q101}
\end{eqnarray}
where
\begin{eqnarray}
\nonumber A^i_1\!\!\!\!&=&\!\!\!\!\!\ 2\beta^7s^i_0-3\lambda q(q-1)\beta^5r_{00}y^i\\
\nonumber A^i_2\!\!\!\!&=&\!\!\!\!\!\ 6\lambda\beta^5s_0y^i-2(3q+2)\beta^6s^i_0+2q(q^2-1)\lambda\beta^4r_{00}y^i\\
\nonumber A^i_3\!\!\!\!&=&\!\!\!\!\!\ 2[(q+1)^2+2q(q+1)]\beta^5s^i_0-2q(6q+1)\lambda \beta^4 s_0y^i+q(q-1)\beta^4r_{00}b^i\\
\nonumber &&+q(q-1)[3\lambda b^2\beta^3r_{00}y^i-2\lambda\beta^4 r_0 y^i],\\
\nonumber A^i_4\!\!\!\!&=&\!\!\!\!\!\ [2q(q+1)(3q-1)-6qb^2]\lambda\beta^3 s_0 y^i+2q(q^2-1)\lambda\beta^2[\beta \nonumber r_0-b^2r_{00}]y^i\\
\nonumber &&-2q\beta^4[(q+1)^2-2b^2]s^i_0-q\beta^3[2\beta s_0+(q^2-1)r_{00}]b^i,\\
\nonumber A^i_5\!\!\!\!&=&\!\!\!\!\!\ 2q(5q+2)\lambda b^2\beta^2 s_0y^i-2q(2q+1)\beta^3[2b^2s^i_0-s_0b^i],\\
\nonumber A^i_6\!\!\!\!&=&\!\!\!\!\!\ 2q^2(q+1)\beta^2[2b^2s^i_0-s_0b^i]-2q^2(3q+1)\lambda b^2\beta s_0y^i\\
&&\nonumber-q^2(q-1)\lambda b^2\beta[2r_0y^i-r_{00}b^i],\\
A^i_7\!\!\!\!&=&\!\!\!\!\!\ 2q^2b^2\beta(b^2s^i_0-s_0b^i),\label{Qw}\\
A^i_8\!\!\!\!&=&\!\!\!\!\!\ -2q^3b^2(b^2s^i_0-s_0b^i).\label{Q9}
\end{eqnarray}
and
\begin{eqnarray}
\nonumber B_0\!\!\!\!&=&\!\!\!\!\!\ 2(q-1)\beta^7\\
\nonumber B_1\!\!\!\!&=&\!\!\!\!\!\ -4(q-1)(q+1)\beta^6\\
\nonumber B_2\!\!\!\!&=&\!\!\!\!\!\ 2(q-1)(q+1)^2\beta^5\\
\nonumber B_3\!\!\!\!&=&\!\!\!\!\!\ 4q(q-1)b^2\beta^4\\
\nonumber B_4\!\!\!\!&=&\!\!\!\!\!\ -4q(q^2-1)b^2\beta^3\\
\nonumber B_5\!\!\!\!&=&\!\!\!\!\!\ 0\\
B_6\!\!\!\!&=&\!\!\!\!\!\ 2q^2(q-1)b^4\beta,\label{Q10}
\end{eqnarray}
and
\begin{eqnarray*}
\bar{A}^i\!\!\!\!&=&\!\!\!\!\!\ \bar{b}^2\bar{s}^i_0-\bar{b}^i\bar{s}_0,\\
\bar{B}^i\!\!\!\!&=&\!\!\!\!\!\ \bar{\beta}[2\lambda y^i(\bar{r}_0+\bar{s}_0)-\bar{b}^i\bar{r}_{00}].
\end{eqnarray*}
(\ref{Q101}) is equivalent to
\begin{eqnarray}
(\sum^8_{j=1}A^i_j\alpha^j)(2\bar{b}^2\bar{\beta})+(\bar{A}^i\bar{\alpha}^2+\bar{B}^i)(\sum^6_{j=0}B_j\alpha^j)
=(2\bar{b}^2\bar{\beta})(\sum^6_{j=0}B_j\alpha^j)H^i_{00}.\label{Q102}
\end{eqnarray}
By replacing $y^i$ with $-y^i$ in (\ref{Q102}) we get
\begin{eqnarray}
\nonumber(\sum^3_{j=0}A^i_{(2j+1)}\alpha^{(2j+1)}\!\!\!\!&-&\!\!\!\!\! \sum^4_{j=1}A^i_{(2j)}\alpha^{(2j)})(-2\bar{b}^2\bar{\beta})\\ \nonumber
\!\!\!\!&-&\!\!\!\!\!(\bar{A}^i\bar{\alpha}^2+\bar{B}^i)(\sum^1_{j=0}B_{(2j+1)}\alpha^{(2j+1)}
-\sum^3_{j=0}B_{(2j)}\alpha^{(2j)})=\\
\!\!\!\!&-&\!\!\!\!\!2\bar{b}^2\bar{\beta}(\sum^1_{j=0}B_{(2j+1)}\alpha^{(2j+1)}
-\sum^3_{j=0}B_{(2j)}\alpha^{(2j)})H^i_{00}.\label{Q103}
\end{eqnarray}
$(\ref{Q102})-(\ref{Q103})$ implies that
\begin{eqnarray}
2\bar{b}^2\bar{\beta}(\sum^4_{j=1}A^i_{(2j)}\alpha^{(2j)})+(\bar{A}^i\bar{\alpha}^2+\bar{B}^i)
\sum^3_{j=0}B_{(2j)}\alpha^{(2j)}
=2\bar{b}^2\bar{\beta}\sum^3_{j=0}B_{(2j)}\alpha^{(2j)}H^i_{00}.\label{Q104}
\end{eqnarray}
$(\ref{Q102})+(\ref{Q103})$ yields
\begin{eqnarray}
\nonumber2\bar{b}^2\bar{\beta}\sum^3_{j=0}A^i_{(2j+1)}\alpha^{(2j+1)}\!\!\!\!&+&\!\!\!\!\!(\bar{A}^i\bar{\alpha}^2+\bar{B}^i)
\sum^1_{j=0}B_{(2j+1)}\alpha^{(2j+1)}\\
 \!\!\!\!&=&\!\!\!\!\! \ 2\bar{b}^2\bar{\beta}\sum^1_{j=0}B_{(2j+1)}\alpha^{(2j+1)}H^i_{00}.\label{Q105}
\end{eqnarray}
By $[\frac{\beta}{q}\times (\ref{Q104})]+[(\ref{Q105})\times\alpha]$ we have
\begin{eqnarray}
&&\nonumber(\bar{A}^i\bar{\alpha}^2+\bar{B}^i)\big[(\frac{\beta}{q})(B_6\alpha^6)+(\frac{\beta}{q}B_4+B_3)\alpha^4
+(\frac{\beta}{q}B_2+B_1)\alpha^2+\frac{\beta}{q}B_0\big]\\
&&+\nonumber(2\bar{b}^2\bar{\beta})\big[(\frac{\beta}{q}A^i_6+A^i_5)\alpha^6+(\frac{\beta}{q}A^i_4+A^i_3)\alpha^4
+(\frac{\beta}{q}A^i_2+A^i_1)\alpha^2\big]=\\
&&(2\bar{b}^2\bar{\beta})H^i_{00}\big[(\frac{\beta}{q})(B_6\alpha^6)+(\frac{\beta}{q}B_4+B_3)\alpha^4
+(\frac{\beta}{q}B_2+B_1)\alpha^2+\frac{\beta}{q}B_0\big].\label{Q107}
\end{eqnarray}
All of member of set $\{(\frac{\beta}{q}A^i_{j}+A^i_{(j-1)}),(\frac{\beta}{q}B_{k}+B_{(k-1)}),\frac{\beta}{q}B_6, \frac{\beta}{q}B_0\hspace{.2cm} j=6,4,2,\hspace{.2cm} k=4,2 \}$ have the factor $\beta^2$. Let us put
\begin{eqnarray}
\nonumber D^i_{j}\!\!\!\!&:=&\!\!\!\!\!\ \frac{1}{\beta^2}(\frac{\beta}{q}A^i_{j}+A^i_{(j-1)}), \ \ \ j=6,4,2\\
\nonumber C_{k}\!\!\!\!&:=&\!\!\!\!\!\ \frac{1}{\beta^2}(\frac{\beta}{q}B_{k}+B_{(k-1)}), \ \ \ k=4,2\\
\nonumber C_{6}\!\!\!\!&:=&\!\!\!\!\!\ \frac{1}{q\beta}B_6=2q(q-1)b^4,\\
C_{0}\!\!\!\!&:=&\!\!\!\!\!\ \frac{1}{q\beta}B_0.
\end{eqnarray}
Then $\frac{1}{\beta^2}\times (\ref{Q107})$ yields
\begin{eqnarray}
\nonumber(\bar{A}^i\bar{\alpha}^2+\bar{B}^i)\!\!\!\!&&\!\!\!\!\! \big[C_6\alpha^6+C_4\alpha^4+C_2\alpha^2+C_0\big]+2\bar{b}^2\bar{\beta}\big[D^i_6\alpha^6+D^i_4\alpha^4+D^i_2\alpha^2\big]
\\ \!\!\!\!&&\!\!\!\!\!=2\bar{b}^2\bar{\beta}H^i_{00}\big[C_6\alpha^6+C_4\alpha^4+C_2\alpha^2+C_0\big].\label{Q108}
\end{eqnarray}
By (\ref{Q104}) and (\ref{Q108}), it follows that $(\bar{A}^i\bar{\alpha}^2+\bar{B}^i)(\sum^3_{j=0}B_{(2j)}\alpha^{(2j)})$ and $(\bar{A}^i\bar{\alpha}^2+\bar{B}^i)\big[C_6\alpha^6+C_4\alpha^4+C_2\alpha^2+C_0\big]$ can be divided by $\bar{\beta}$. Thus $\beta=\mu\bar{\beta}$ and $\bar{A}^i\bar{\alpha}^2 C_6 \alpha^4$ can be divided by $\bar{\beta}$. Since $\bar{\beta}$ is prime with respect to $\alpha$ and $\bar{\alpha}$, then $\bar{A}^i:=\bar{b}^2\bar{s}^i_0-\bar{b}^i\bar{s}_0$ can be divided by $\bar{\beta}$. Hence, there is a scaler function $\psi^i(x)$ such that
\begin{eqnarray}
\bar{b}^2\bar{s}^i_0-\bar{b}^i\bar{s}_0=\psi^i\bar{\beta}.\label{Q010}
\end{eqnarray}
Contracting (\ref{Q010}) with $\bar{y}_i:=\bar{a}_{ij}y^j$  yields $\psi^i(x)=-\bar{s}^i$. Then we have
\begin{eqnarray}
\bar{s}_{ij}=\frac{1}{\bar{b}^2}(\bar{b}_i\bar{s}_j-\bar{b}_j\bar{s}_i).\label{OO1}
\end{eqnarray}
Now, suppose that $(n\geq 3)$. Then by Lemma \ref{lem1}, $\bar{F} =\frac{\bar{\alpha}^2}{\bar{\beta}}$ is a Douglas metric. Since $F$ and $\bar{F}$ have the same Douglas tensor, both of them are Douglas metrics.

If $(n=2)$, $\bar{F}=\frac{\bar{\alpha}^2}{\bar{\beta}}$ is a Douglas metric by Lemma \ref{lem1}. Thus $F$ and $\bar{F}$ having  the same Douglas tensor means that they are all Douglas metrics. This completes the
 proof of Lemma $\ref{lemma2}$
\end{proof}

\bigskip

Now, we are in the position to prove Theorem $\ref{mainthm3}$.

\bigskip

\noindent
{\bf Proof of Theorem \ref{mainthm1}:}  First, we proof the necessity. If $F$ is projectively related to $\bar{F}$, then  they have the same Douglas tensor. By Lemma $\ref{lemma1}$, $F$ and $\bar{F}$
are both Douglas metrics. If $q=1,-1$, then it is easy to prove that $\phi(s)=\frac{s^q}{(s-1)^{q-1}}$ satisfies $\frac{Q}{S}\neq constant$. By Lemma $\ref{lemma3}$, it results that $s_{ij}=0$. By ($\ref{G1}$),  we get
\begin{eqnarray}
\nonumber G^i=G^i_{\alpha}\!\!\!\!&+&\!\!\!\!\!\ \frac{1}{2}\frac{\beta(\beta-2q\alpha)r_{00}}{\beta^2(\beta-\alpha)+q(b^2\alpha^2-\beta^2)\alpha} y^i\\
\!\!\!\!&+&\!\!\!\!\!\ \frac{1}{2}\frac{ q\alpha^3 r_{00}}{\beta^2(\beta-\alpha)+q(b^2\alpha^2-\beta^2)\alpha} b^i.\label{G1}
\end{eqnarray}
Plugging ($\ref{OO1}$) and (\ref{Q03}) into ($\ref{GG01}$) yields
\begin{eqnarray}
\bar{G}^i=\bar{G}^i_{\bar{\alpha}}-\frac{1}{2\bar{b}^2}\big[-\bar{\alpha}^2\bar{s}^i+
(2\bar{s}_0y^i-\bar{r}_{00}\bar{b}^i)+2\frac{\bar{r}_{00}\bar{\beta}y^i}{\bar{\alpha}^2}\big].\label{GG0}
\end{eqnarray}
By assumption, there is a scalar function $P = P(x,y)$ on $TM_0$ such that $G^i=\bar{G}^i+Py^i$. Then by (\ref{G1}) and (\ref{GG0})  we have
\begin{eqnarray}
\nonumber\Big[P\!\!\!\!&-&\!\!\!\!\!\ \frac{1}{2}\frac{\beta(\beta-2q\alpha)r_{00}}{\beta^2(\beta-\alpha)+q(b^2\alpha^2-\beta^2)\alpha}
-\frac{1}{2\bar{b}^2}(\bar{s}_0+\bar{r}_{00}\bar{b}^i)\Big]y^i\\
\!\!\!\!&=&\!\!\!\!\!\ G^i_{\alpha}-\bar{G}^i_{\bar{\alpha}}-\frac{1}{2\bar{b}^2}(\bar{\alpha}^2\bar{s}^i+\bar{r}_{00}\bar{b}^i)
+\frac{1}{2}\frac{ q\alpha^3 r_{00}}{\beta^2(\beta-\alpha)+q(b^2\alpha^2-\beta^2)\alpha} b^i.\label{D110}
\end{eqnarray}
The right side of (\ref{D110}) is a quadratic in $y$. Then there exists a 1-form
$\theta=\theta_i(x)y^i$ on $M$ such that
\begin{eqnarray}
P-\frac{1}{2}\frac{\beta(\beta-2q\alpha)r_{00}}{\beta^2(\beta-\alpha)
+q(b^2\alpha^2-\beta^2)\alpha}-\frac{1}{2\bar{b}^2}(\bar{s}_0+\bar{r}_{00}\bar{b}^i)=\theta.
\end{eqnarray}
Thus we get
\begin{eqnarray}
G^i_{\alpha}+\frac{1}{2}\frac{ q\alpha^3 r_{00}}{\beta^2(\beta-\alpha)+q(b^2\alpha^2-\beta^2)\alpha} b^i=\bar{G}^i_{\bar{\alpha}}+\frac{1}{2\bar{b}^2}(\bar{\alpha}^2\bar{s}^i+\bar{r}_{00}\bar{b}^i)+\theta y^i.\label{D1110}
\end{eqnarray}
This completes the proof of the necessity.

Conversely, from (\ref{G1}), (\ref{GG0}) and (\ref{p7}) we have
\begin{eqnarray}
G^i=\bar{G}^i+\big[\theta+\frac{1}{2}\frac{\beta(\beta-2q\alpha)r_{00}}{\beta^2(\beta-\alpha)
+q(b^2\alpha^2-\beta^2)\alpha}+\frac{1}{2\bar{b}^2}(\bar{s}_0+\bar{r}_{00}\bar{b}^i)\big]y^i.
\end{eqnarray}
Thus $F$ is projectively equivalent to $\bar{F}$. This completes the proof.
\qed

\bigskip

By Lemma \ref{lem1}, Lemma \ref{lemma3} and Theorem \ref{mainthm3}, we have the following.
\begin{cor}\label{mainthm4}
Let $F=\frac{\beta^q}{(\beta-\alpha)^{q-1}}$ $(q\neq 1,-1)$  be a $(q,\alpha,\beta)$-metric and $\bar{F}=\frac{\bar{\alpha^2}}{\bar{\beta}}$  be a Kropina metric on a $n$-dimensional
manifold $M$ $(n \geq3)$ where $\alpha$ and $\bar{\alpha}$ are two Riemannian metrics, $\beta$ and $\bar{\beta}$ are two nonzero collinear 1-forms. Then $F$ is projectively related to $\bar{F}$ if and only if  the following holds
\begin{eqnarray}
G^i_{\alpha}-\bar{G}^i_{\bar{\alpha}}\!\!\!\!&=&\!\!\!\!\!\frac{1}{2\bar{b}^2}(\bar{\alpha}^2\bar{s}^i+\bar{r}_{00}\bar{b}^i)+\theta y^i-\frac{1}{2}\frac{ q\alpha^3 r_{00}}{\beta^2(\beta-\alpha)+q(b^2\alpha^2-\beta^2)\alpha} b^i,\\
s_{ij}\!\!\!\!&=&\!\!\!\!\!0,\\
\bar{s}_{ij}\!\!\!\!&=&\!\!\!\!\! \frac{1}{\bar{b}^2}\{\bar{b}_i\bar{s}_j-\bar{b}_j\bar{s}_i\}.
\end{eqnarray}
where $b_{i|j}$ denote the coefficients of the covariant derivatives of $\beta$ with respect to $\alpha$.
\end{cor}

%-----------------------------------------------------------------------------------

\bigskip

\noindent
Akbar Tayebi and Hassan Sadeghi\\
Department of Mathematics, Faculty  of Science\\
University of Qom \\
Qom. Iran\\
Email:\ akbar.tayebi@gmail.com

\bigskip

\noindent
Esmaeil Peyghan\\
Department of Mathematics, Faculty  of Science\\
Arak University\\
Arak 38156-8-8349,  Iran\\
Email: epeyghan@gmail.com
\end{document}